\documentclass{amsart}

\usepackage{amsmath,amssymb,latexsym,amsfonts,amsthm,mathrsfs}
\usepackage{mathtools}
\usepackage{url}

\theoremstyle{plain}
\newtheorem{Thm}{Theorem}[section]

\newtheorem{Lem}[Thm]{Lemma}

\newtheorem{Ass}[Thm]{Assumption}
\theoremstyle{definition}

\newtheorem{Rem}[Thm]{Remark}

\newcommand{\dimX}{l}
\newcommand{\dimY}{m}
\newcommand{\dimB}{d}
\newcommand{\dimZ}{ {\dimY \times \dimB} }

\title{
Fully coupled drift-less Forward and backward Stochastic Differential Equations in a degenerate case
}
\author{Takahiro Tsuchiya}

\begin{document}
\maketitle

\begin{abstract}
Existence and uniqueness results of fully coupled forward stochastic differential equations without drifts 
and backward stochastic differential equations 
in a degenerate case are obtained for an arbitrarily large time duration. 
\end{abstract}
 
\section{Introduction}

Let $(\Omega, \mathscr { F }, \mathbb{P} )$ be a probability space, and let $\left\{\mathrm{W} (t)\right\}_{t \geq 0}$ be a $d$-dimensional Wiener process in this space. We denote the natural filtration of this Wiener process by $\mathscr { F }_{t} .$ In this paper, we consider the following fully coupled forward-backward stochastic differential equation (FBSDE); for arbitrary fixed $(t, x) \in [0,T] \times \mathbb{R}^{\dimX}$, 
\begin{equation}\label{drifted FBSDEs}
	\begin{dcases}
		X(r) &= x 
		+ \int_{t}^{r} b \left( s, X(s), Y(s), Z(s) \right) \mathrm{d}s
		+ \int_{t}^{r} \sigma \left( s, X(s), Y(s), Z(s) \right) \mathrm{dW}(s), 
		\\ Y(r) &= \varphi \left( X (T)\right) + \int_{r}^{T} f \left(  s, X(s), Y(s), Z(s) \right) \mathrm{d}s -  \int_{r}^{T} Z (s) \mathrm{dW}(s), \quad r \in [t, T]. 
	\end{dcases}
\end{equation}
where $(X, Y, Z)$ takes values in $\mathbb{R}^{\dimX} \times \mathbb{R}^{\dimY} \times \mathbb{R}^{\dimZ}$, and $b$, and $\sigma$, and $f$ are mappings with appropriate dimensions which are, for each fixed $(s, X(s), Y(s), Z(s)), \mathscr{F}_{s}$-progressively measurable for $t \leq s \leq T$. We assume that they are Lipschitz with respect to the spatial variable $(x, y, z)$ ; $T>0$ is an arbitrarily prescribed number and the time interval is called the time duration. We look for a triple of $\mathscr{F}_{r}$-adapted processes $\left(X(r), Y(r), Z(r)\right)$ satisfying this equation \eqref{drifted FBSDEs}. 

It is known that for a sufficient small duration, the local existence and uniqueness holds using a contraction map \cite{antonelli1993} and \cite{Delarue2002209} under a Lipschitz condition. 
Indeed, it was shown by Antonelli in \cite{antonelli1993} that the Lipschitz condition is not enough for the existence and uniqueness of FBSDE in an arbitrarily large time duration. 

Therefore, it has made some significant progress in the fully coupled FBSDE in the difference view points: 
a kind of four-steps scheme approach \cite{Ma1994}, \cite{Pardoux1999}, 
the method of continuation and a monotonicity condition, \cite{Hu1995}, \cite{Yong1997} and \cite{Peng1999}.  
Recently, a different approach have been proposed in \cite{Ma2015} to study the decoupling field $u(t, x)$. 
They also pointed out one of the most important property that as the well-poshness of FBSDEs is essentially given by a so called {\it Characteristic BSDE} whose generator has at least quadratic growth in both $Y$ and $Z$. 

In this paper, we present a probabilistic and an analytic method to treat the fully coupled forward drift-less and backward SDE with a degenerate case formally introduced as \eqref{ass: degeneration} in Assumption \ref{ass: all condition}; for arbitrary fixed $(t, x) \in [0,T] \times \mathbb{R}^{\dimX}$
\begin{equation}\label{drifted-less FBSDEs}
	\begin{dcases}
		X(r) &= x 
		+ \int_{t}^{r} \sigma \left( s, X(s), Y(s), Z(s) \right) \mathrm{dW}(s), 
		\\ Y(r) &= \varphi \left( X (T)\right) + \int_{r}^{T} f \left(  s, X(s), Y(s), Z(s) \right) \mathrm{d}s -  \int_{r}^{T} Z (s) \mathrm{dW}(s), \quad r \in [t, T]. 
	\end{dcases}
\end{equation}
It provides that there exists a pair of continuous function $(u, v)$ such that 
it is $(1/2)$-H\"older continuous with respect to the spatial variable and 
it is consistent with so called Backward Stochastic Riccati Equations.
Furthermore, 
it satisfies 
\begin{equation*}
	\sigma (s, X(s), Y(s), Z(s)) = \sigma (s, X(s), u(s, X(s)), v(s, X(s))), \quad t \leq s \leq T.
\end{equation*}
Then, the forward SDE become a unique strong solution 
with the $(1/2)$-H\"older continuous diffusion coefficient which is critical condition in the sense of Yamada-Watanabe condition in \cite{YW} for one dimension. 

The advantages of our method are as follows: 
 (i) the assumptions are very simple and natural to verify and we can treat the case without a monotonicity condition. 
 (ii) Many existing problems of FBSDE in mathematical finance and machine learning (\cite{Ho2020NEURIPS2020_4c5bcfec, song2021scorebased}) satisfy these assumptions. Furthermore, as a degenerate case is strongly related to {\it vanishing gradient problem}, the theoretical development have been expected. 
 (iii) We do not need to impose the non degenerate condition on the diffusion term. 
 This allows us to treat FBSDEs which the previous studies does not cover.


The paper is organized as follows. 
For preliminaries, we introduce necessary notations and assumptions in section \ref{sec:preliminaries}. 
In section \ref{sec:Continuity of decoupling fields}, 
we define an approximation process and introduce a pair. 
We show that the pair is uniformly equicontinuous and satisfies linear growth condition.  
Therefore, we can find a sub-sequence which uniformly converges. 
In section \ref{sec:Global solutions}, 
the uniform convergent property implies a stability problem of the forward SDEs. 
Then, we obtain the global solution as the convergence point. 
Independently, we set a section \ref{sec:Martingale estimation}. 
In this section, a necessary important estimation is shown. 
The transition function of SDE with uniformly bounded coefficients, has a uniformly equicontinuous for the time and space, which is obtained an analytical approach known as Calder\'on Zygmund lemma. 
In Lemma \ref{lem: a key lemma: Lipschitz version}, we show a simple but convent inequality of the structure without the uniformly bounded assumption.


\section{Preliminaries}\label{sec:preliminaries}
\renewcommand{\theenumi}{A.\arabic{enumi}}
\renewcommand{\labelenumi}{(\theenumi)}

Let $\mathrm{W}$ be a standard Wiener process with values in $\mathbb{R}^{\dimB}$ defined on some complete probability space $(\Omega, \mathscr { F }, \mathbb{P} )$. $\{ \mathscr { F }_t \}_{t \geq 0}$ is an argument of natural filtration of $\mathrm{W}$ which satisfies usual condition. $\mathbb { R } ^ { \dimZ }$ is identified with the space of real matrices with $\dimY$ rows and $\dimB$ columns. If $z \in \mathbb{R}^{\dimZ}$, we have $| z | ^ { 2 } = \operatorname { trace } \left( z z ^ { * } \right)$ where $| \cdot| $ stands for the Frobenius norm. 


For any real $\dimX \in \mathbb{N}$ and $T>0$, $ \mathscr{ S }^2 \left( \mathbb { R } ^ { \dimX } \right)$, denotes the set of $\mathbb{R}^{\dimX}$-valued, adapted and  cadlag process $\left\{ X (t) \right\} _ { t \in [ 0 , T ] }$ such that 
	$\| X \| = \left\{ \mathbb { E } \left[ \sup_{ 0 \leq t \leq T} \left| X (t) \right| ^ { 2 } \right]\right\} ^ { \frac{1}{2} } < + \infty$. 
A collection $\mathscr{H}  ^ { 2 } \left( \mathbb { R } ^ { \dimZ } \right)$ denotes the set of (equivalent classes of) predictable processes $\left\{ Z (t) \right\} _ { t \in [ 0 , T ] }$ with values in $\mathbb{R}^{\dimZ}$ such that 
	$\| Z \|  = \left\{ \mathbb { E } \left[ \left( \int _ { 0 } ^ { T } \left| Z _ { r } \right| ^ { 2 } \mathrm { d } r \right) ^ { } \right] \right\}^ { \frac{1}{2} } < + \infty$. 
We write a Banach space $	\mathscr { S } ^ { 2 } (\mathbb{R}^{\dimX}) \times \mathscr { S } ^ { 2 } (\mathbb{R}^{\dimY}) \times \mathscr{H}  ^ { 2 } \left( \mathbb { R } ^ { \dimZ } \right) \equiv \mathscr{S}^2 \times \mathscr{S}^2 \times \mathscr{H}^2$ if there is no risk to confuse. For $(X, Y, Z) \in \mathscr{S}^2 \times \mathscr{S}^2 \times \mathscr{H}^2$, we note that 
	\begin{equation*}
	\begin{split}
	(X, Y, Z, \mathrm{W}): [0, T] \times \Omega \longrightarrow  \mathbb{R}^{\dimX} \times \mathbb{R}^{\dimY} \times \mathbb{R}^{\dimZ} \times \mathbb{R}^{\dimB}.
	\end{split}
	\end{equation*}
	For $\forall t \in [0,T]$, $\forall (x, y, z) \in \mathbb{R}^{\dimX} \times \mathbb{R}^{\dimY}\times \mathbb{R}^{\dimZ} $, we consider measurable functions $(x, y, z) \mapsto \sigma (t, x, y, z)$, $(x, y, z) \mapsto f (t, x, y, z)$ and $x \mapsto \varphi (x) $. For any given initial distribution $\mu$, we suppose that it holds that \label{ass: initial}
		\begin{equation*}
		\int_{\mathbb{R}^{\dimX}}  \int_{0}^{T} 
		|\sigma (s, x, \varphi (x) , 0 )|^2 +|\varphi (x)|^2 + 
		| f (s, x, \varphi (x), 0 )|^2 \mathrm{d}s \mu (\mathrm{d}x )	< \infty. 
		\end{equation*}

	\begin{Ass}\label{ass: all condition}
	We say that Assumption \ref{ass: all condition} holds if  
	\begin{enumerate}
		\item The functions $(x, y, z) \mapsto \sigma (t, x, y, z)$, $(x, y, z) \mapsto f (t, x, y, z)$ and $x \mapsto \varphi (x) $ are infinitely differentiable with uniformly bounded derivatives. \label{ass: smooth}
		\item There exists a constant $\Lambda$ such that $\forall t \in [0,T]$, $\forall (x, y, z) \in \mathbb{R}^{\dimX} \times \mathbb{R}^{\dimY}\times \mathbb{R}^{\dimZ} $, \label{ass: holder}
		\begin{equation*}
			\begin{split}
			|\varphi (x+h)- \varphi (x)| &\leq \Lambda |h|^{1/2}, 
			\\	|f (t, x+h, y, z)-f (t, x, y, z)| &\leq \Lambda |h|^{1/2}, \quad h \in \mathbb{R}^{\dimX}, \ 
				|h|>1.
			\end{split}
		\end{equation*} 
		\item The following set $\mathcal{Z}_{0}$ is not empty set: \label{ass: degeneration}
		\begin{equation*}
		\begin{split}
			\mathcal{Z}_0 :=\left\{ x_0 \in \mathbb{R}^{\dimX}: 
			\sigma \left(s,  x_0, \varphi (x_0), 0 \right) = f \left(s, x_0, \varphi (x_0), 0 \right) =0, \ s \in [0,T]
			\right\} \neq \emptyset. 
		\end{split}	
		\end{equation*}
	\end{enumerate}
	\end{Ass}
	
	For the functions $\sigma $ and $\varphi$,  we denote
	  	\begin{equation*}
	\begin{split}
		&L_{\varphi, x} \triangleq \inf \left\{ L >0 : 
		\forall x_i \in \mathbb{R}^{\dimX}\  (i=1, 2), \quad | \varphi ( x_1 ) - \varphi ( x_2 ) | 
		\leq L |x_1 -x_2| 
		\right\}, 
		\\
		&	L_{\sigma, z} \triangleq \inf \left\{ L >0 : 
		\forall (t, x, y, z_i) \in [0,T] \in \mathbb{R}^{\dimX} \times \mathbb{R}^{\dimY}\times \mathbb{R}^{\dimZ} \  (i=1,2), \right.
		\\ & \left. \quad \quad \quad \quad \quad \quad \quad \quad \quad\quad \quad \quad  | \sigma (t, x, y, z_1) -  \sigma (t, x, y, z_2) | 
		\leq L |z_1 -z_2| 
		\right\},
	\end{split}
	\end{equation*}	
	and $L_{\sigma, x}$, $L_{\sigma, y}$ are defined by the same manner. 
	
	\subsection{Decoupled FBSDEs flow: $(X^{t, x}, Y^{t, x}, Z^{t, x})$}
	Under \eqref{ass: smooth}, 
	for each $t \in[0, T), x \in \mathbb{R}^{\dimX}$, we denote by $X^{t, x} = \left\{X^{t, x}(r), t \leq r \leq T\right\}$ the unique strong solution of the following SDE:
\begin{equation*}
\left\{\begin{aligned}
	\mathrm{d} X^{t, x}(u) &=
	\sigma\left(s, X^{t, x}(s)\right) \mathrm{d} \mathrm{W}(s), \quad t \leq s \leq T \\
	X^{t, x}(t) &=x,
\end{aligned}\right.	
\end{equation*}
or, it is equivalent to 
\begin{equation*}
	X^{t, x}(r) =x 
	+ \int_{t}^{r} \sigma\left(s, X^{t, x}(s)\right) \mathrm{d} \mathrm{W}(s), \quad t \leq r \leq T .
\end{equation*}
Moreover, the random field $\left\{X^{t, x}(r) ; \ 0 \leq t \leq r \leq T,\  x \in \mathbb{R}^{\dimX}\right\}$ has a version which is a.s.~jointly continuous in $(t, s, x)$, together which its $x$ infinity partial derivatives, see e.g.~\cite{IkedaWatanabe}. 

For any $t \in[0, T)$ and $x \in \mathbb{R}^{d}$, let $(Y^{t, x}, Z^{t, x}) =  \left\{\left(Y^{t, x} (r) , Z^{t, x} (r)\right) ; t \leq r \leq T\right\}$ denote the unique element of $\mathscr{S}^2 \times \mathscr{H}^2$ which solves the following BSDE : for all $ t \leq r \leq T$, 
\begin{equation*}
Y^{t, x} (r) =\varphi \left(X_{T}^{t, x}\right)+\int_{r}^{T} f\left(s, X^{t, x}(s), Y^{t, x}(s), Z^{t, x}(s)\right) \mathrm{d} s-\int_{r}^{T} Z^{t, x}(s) \mathrm{d}\mathrm{W}(s) . 
\end{equation*}
In short, $(X^{t, x}, Y^{t, x}, Z^{t, x})$ is the unique solution to the decoupled FBSDE with the initial condition $X^{t, x}(t) =x $.

	\begin{Rem}
	The condition \eqref{ass: holder} does not include the previous result in \cite{Tsuchiya2021v2} 
	in the case of one dimension. 
	\end{Rem}

\section{Continuity of decoupling fields}\label{sec:Continuity of decoupling fields}
Let us explain briefly our idea. According to the proof of Lemma 2.5.15 in \cite{FrommPhD}, if a strongly regular and deterministic decoupling field exists, it satisfies 
	\begin{equation*}
		\left| u(t+h, x) - u(t, x)  \right| \leq 
		C \left( 1+  \mathbb{E} \left[ \sup_{ t \leq s \leq t+h }\left|  X^{t, x} (s)\right| \right]\right) |h|, \quad t, \ t+h \in [t,T]. 
	\end{equation*}
Moreover, there exists an example such that for $T>1$, 
	\begin{equation*}
		\lim_{ t \downarrow T-1 } \mathbb{E} \left[ \sup_{ t \leq s \leq t+h }\left|  X^{t, x} (s)\right| \right]  
		=\lim_{ t \downarrow T-1 } L_{u (t, \cdot), x}
		= \infty, 
	\end{equation*} 
	and we will see at this example later in Remark \ref{lem:Existence for ODEs and SDEs}. 
From the observation, we may say that the backward solution and the decoupling fields may be bounded by the forward process. Therefore, it leads to construct a forward process based approximation scheme as follows.

\subsection{Construction of the scheme via $(u_n, v_n)$}
Under the condition that all coefficients are regular enough, 
we consider a solution of SDEs and then we introduce the corresponding solution of BSDEs. 
Thus, we obtain the solution to the decoupling FBSDEs,. 
Thanks to \cite{PardouxPeng1992}, it leads a classical smooth solution  to a corresponding semi-linear PDEs. 
We note that the proof is also applied to the situation $\dimY \geq 1$. 

More preciously, we define as follows: Suppose \eqref{ass: smooth} holds. 
We shall define for $(t, x)\in [0,T] \times \mathbb{R}^{\dimX}$, 
\begin{equation*}
	\left( X^{t, x}_{0} (r), Y^{t, x}_0 (r) , Z^{t, x}_{0}(r) \right) =(x, \varphi (x), 0), \quad r \in [t, T],
\end{equation*}
and 
\begin{equation*}
	u_{0}(t, x)= \varphi ( x ) , \quad v_{0}(t, x)=0, \quad (t, x) \in [0,T] \times \mathbb{R}^{\dimX}. 
\end{equation*}
Then, we define the following approximation series, $\Theta_{n-1}^{t, x} := \left( X_{n-1}^{t, x}, Y_{n-1}^{t, x}, Z_{n-1}^{t, x}\right)$ 
for $(t, x, n) \in [0,T]\times \mathbb{R}^{\dimX}\times \mathbb{N}$ such that it satisfies 
\begin{equation}\label{decoupling FBSDE approximation via u and v}
\begin{split}&
	\begin{dcases}
	X^{t, x}_{n} (r) &= x + \int_{t}^{r} \sigma \left( s, X^{t, x}_{n}(s), u_{n-1} (r, X^{t, x}_{n} (s) ), v_{n-1} (r, X^{t, x}_{n} (s) ) \right) \mathrm{dW}(s)
	\\ Y^{t, x}_{n-1} (r) &= \varphi \left( X^{t, x}_{n-1} (T) \right) + \int_{r}^{T} f \left( s, X^{t, x}_{n-1}(s), Y^{t, x}_{n-1}(s), Z^{t, x}_{n-1}(s) \right) \mathrm{d}s
	\\&\quad -  \int_{r}^{T}  Z^{t, x}_{n-1}(s) \mathrm{dW}(s), \quad (t, x) \in [0,T] \times \mathbb{R}^{\dimX}, 
	\end{dcases}
\end{split}
\end{equation}
where, $u_{n-1}$ is a solution to a semi-linear parabolic PDE solution and it holds that 
\begin{equation*}
	Y^{t, x}_{n-1}(r) = u_{n-1} (r, X^{t, x}_{n-1} (r) ), \quad  Z^{t, x}_{n-1} (r)= v_{n-1} (r, X^{t, x}_{n-1} (r) ), \quad (n, r) \in \mathbb{N} \times [t, T], 
\end{equation*}
where $v_{n-1} (t, x) = \nabla_x u_{n-1}(t, x) \cdot \sigma (t, x, u_{n-1}(t, x), v_{n-1}(t, x))$ for all $(t, x) \in [0,T]\times \mathbb{R}^{\dimX}$. 
For the detail, see Lemma \ref{lem: an expression via FK and BSDE uniqueness}. 
It is remarkable that generally we have 
\begin{equation*}
	Y^{t, x}_{n} (r) =u_{n}(r, X^{t, x}_{n}(r)) \neq u_{n}(r, X^{t, x}_{n+1}(r)), \quad 
	Z^{t, x}_{n} (r) =v_{n}(r, X^{t, x}_{n}(r))  \neq v_{n}(r, X^{t, x}_{n+1}(r)).
\end{equation*}



\subsection{A priori estimation and Degenerate condition}
Firstly, we recall a priori estimate of BSDEs that the $\mathscr{H}^2$-norm of the solution can be bounded by the final point and the random driver; 
\begin{equation*}
	\begin{split}
\exists C>0; \quad 
	&\mathbb{E}\left[ \int_{0}^{T} | Y_1 (s) - Y_2 (s) |^2 + | Z_1 (s) - Z_2 (s) |^2 \mathrm{d}s \right] 
	\\&\leq C
	\mathbb{E} \bigg[ \varphi \left( X_1 (T) \right)-\varphi \left( X_2 (T) \right) | ^2 \biggr. 
	\\& \left.
	+ \int_{0}^{T} |f(s, X_1 (s), Y_2 (s), Z_2 (s))-f(s, X_2 (s), Y_2 (s), Z_2 (s))|^2 \mathrm{d}s \right]. 
	\end{split}
\end{equation*} 
This says that the squared norm of the difference of $(Y, Z)$ can be estimated by that of the forward process. 
Moreover, considering the super-linear growth of the characteristic BSDEs, we consider \eqref{ass: holder}. 
This allows us to obtain a constant $C>0$ such that 
\begin{equation*}
\begin{split}
	&\mathbb{E} \left[ |\varphi \left( X_1 (T) \right)-\varphi \left( X_2 (T) \right) | ^2 \right]
	\leq C \sup_{s \in [0,T]}\mathbb{E}\left[  \left| X_1(s) -X_2 (s)\right| \right],	
	\\ 
	&\mathbb{E} \left[ |f(s, X_1 (s), Y_2 (s), Z_2 (s))-f(s, X_2 (s), Y_2 (s), Z_2 (s))|^2 \right]
	\\& \quad \leq C \sup_{s \in [0,T]}\mathbb{E}\left[  \left| X_1(s) -X_2 (s)\right| \right].	
\end{split}
\end{equation*}
Again applying the exponential inequality, Lemma \ref{lem: a key lemma: Lipschitz version}, we can estimate the above $L^1$-norm. Furthermore, we note that, as one may consider $L^p$-norm estimate, to best of our knowledge, it seems to be hard to estimate the backward as above except for $p=1$. 

\begin{Lem}[Equicontinuous $u_n$ with respect to spatial variables]\label{lem:equicontinuous un}
Under \eqref{ass: smooth} and \eqref{ass: holder},  
there exists a constant $C>0$ such that for all $n \in \mathbb{N} \cup \{ 0\}$ satisfies 
\begin{equation*}
	\sup_{n  \in \mathbb{N}}\sup_{x \in \mathbb{R}^{\dimX}}\sup_{0 \leq t \leq T}| u_{n}(t, x+h)-u_{n}(t, x) |  \leq  C |h|^{1/2}.
\end{equation*}
\end{Lem}
\begin{proof}
Let $n \in \mathbb{N} \cap \{ 0\}$. 
By the definition, we have $Y_n^{t, x} (r) = u_{n}(r, X_{n}^{t, x}(r))$ and it holds for any $(t, x) \in [0,T] \times \mathbb{R}^{\dimX}$, 
\begin{equation*}
\begin{split}
	 Y_n^{t, x} (r) &= Y_n^{t, x} (T) + \int_{r}^{T} f \left(s, \Theta_n^{t, x} (s) \right)\mathrm{d}s- \int_{r}^{T} Z_n^{t, x} (s) \mathrm{dW}(s), \quad t \leq r \leq T. 	
\end{split}
\end{equation*}
It leads to 
\begin{equation*}
\begin{split}
	\left\{ Y_n^{t, x} (r) + \int_{t}^{r} f \left(s, \Theta_n^{t, x} (s) \right)\mathrm{d}s\right\}_{t \leq r \leq T}
\end{split}
\end{equation*}
is a martingale. 
Thus, we obtain 
\begin{equation*}
	u_{n}(t, x) = \mathbb{E} \left[\varphi (X_{n}^{t, x}(T)) + \int_{t}^{T} f(s, \Theta_n^{t, x} (s)) \mathrm{d}s \right], \quad (n, t, x) \in \mathbb{N}\times[0,T]\times \mathbb{R}^{\dimX}. 
\end{equation*}
Notify that for all $ (n, t, x) \in \mathbb{N}\times[0,T]\times \mathbb{R}^{\dimX}$, it holds that 
\begin{equation*}
	\begin{split}
	&u_{n}(t, x+h)- u_{n}(t, x)
	\\&= \mathbb{E} \left[\varphi (X_{n}^{t, x+h}(T))- \varphi (X_{n}^{t, x}(T))  
	+ \int_{t}^{T} f(s, \Theta_n^{t, x+h} (s)) - f(s, \Theta_n^{t, x} (s)) \mathrm{d}s \right]	
	\end{split}
\end{equation*}
Let us consider the driver term. It follows from the Lipschitz continuous \eqref{ass: smooth} that 
there exists a constant $C>0$ such that 
\begin{equation*}
	\begin{split}
	&\int_{t}^{T} \mathbb{E} [ \left| f(s, \Theta_n^{t, x+h} (s)) - f(s, \Theta_n^{t, x} (s)) \right|^2 ]	\mathrm{d}s
	\\&\leq C \int_{t}^{T} \mathbb{E} \left[ \left| Y^{t, x +h}_n (s) - Y^{t, x }_n (s)\right|^2 \right]\mathrm{d}s 
	+ C \int_{t}^{T} \mathbb{E} \left[ \left| Z^{t, x +h}_n (s) - Z^{t, x }_n (s)\right|^2 \right]\mathrm{d}s 
	\\&+ C \int_{t}^{T} \mathbb{E} \left[ \left| f (s, X^{t, x+h }_n (s), Y^{t, x }_n (s), Z^{t, x }_n (s))  - f (s, X^{t, x }_n (s), Y^{t, x }_n (s), Z^{t, x }_n (s)) \right|^2 \right]\mathrm{d}s 
	\end{split}
\end{equation*}
As we have a dominated property formally provided by Lemma \ref{lem: Y dominated by X}, it implies that there exists a constant $C>0$ such that 
\begin{equation*}
	\begin{split}
	&\int_{t}^{T} \mathbb{E} [ \left| f(s, \Theta_n^{t, x+h} (s)) - f(s, \Theta_n^{t, x} (s)) \right|^2 ]	\mathrm{d}s
	\\&\leq C \, \mathbb{E} [ \left| \varphi (X^{t, x+h }_n (T)) -\varphi (X^{t, x }_n (T))\right|^2 ]
	\\&+ C \int_{t}^{T} \mathbb{E} \left[ \left| f (s, X^{t, x+h }_n (s), Y^{t, x }_n (s), Z^{t, x }_n (s))  - f (s, X^{t, x }_n (s), Y^{t, x }_n (s), Z^{t, x }_n (s)) \right|^2 \right]\mathrm{d}s.
	\end{split}
\end{equation*}
From the squared condition \eqref{ass: holder}
we can find some constant $\Lambda > 0$ such that 
	\begin{equation*}
		\begin{split}
		&
		 \mathbb{E}\left[ \left| \varphi (X^{t, x+h }_n (T)) -\varphi (X^{t, x }_n (T))\right|^2 \right.
		 \\&+ \left.
			\int_{t}^{T} \left| f (s, X^{t, x+h }_n (s), Y^{t, x }_n (s), Z^{t, x }_n (s))  - f (s, X^{t, x }_n (s), Y^{t, x }_n (s), Z^{t, x }_n (s))  \right|^2 \mathrm{d}s\right]
		\\& \leq \Lambda \mathbb{E}\left[	 \left|\varphi \left( X^{t, x+ h  }_n (T)\right) - \varphi \left( X^{t, x  }_n (T) \right)\right| 
		 +	\int_{t}^{T} |X^{t, x+ h  }_n (s) - X^{t, x  }_n (s)| \mathrm{d}s\right],
		\end{split}
	\end{equation*}
	see also Remark \ref{rem:Uniformly bounded Lipschitz continuous}. 
Thus, we have  a constant $C>0$ such that 
\begin{equation*}
\begin{split}
	&
	| u_{n}(t, x+h)-u_{n}(t, x) |^2  
	\\&\leq C 
	\left\{
	 \mathbb{E}[ |X^{t, x+ h  }_n (T) - X^{t, x  }_n (T)| ] 
	 +\int_{t}^{T} \mathbb{E}[ |X^{t, x+ h  }_n (s) - X^{t, x  }_n (s)| ]   \mathrm{d}s
	 \right\}.	
\end{split}
\end{equation*}
Then, it follows from Lemma \ref{lem: a key lemma: Lipschitz version} that there exists a positive constant $C$ such that
\begin{equation*}
	\sup_{x \in \mathbb{R}^{\dimX}}\sup_{0 \leq t \leq T} \sup_{n \in \mathbb{N}}\mathbb{E}[ |X^{t, x+ h  }_n (s) - X^{t, x  }_n (s)| ] \leq C |h|.
\end{equation*}
Therefore, we conclude that $u_{n}$ is $(1/2)$-Holder continuous with respect to spatial variables.

\end{proof}

\subsection{How to estimate $v_n$}
Roughly, $\{ v_n \}$ is given by the following implicit functions; for any $n \in \mathbb{N}$, 
\begin{equation*}
	\begin{split}
	v_{n} (s, x) &=  \left( \nabla_x u_{n} (s, x) \right) \, \sigma \left(s, x, u_{n} (s, x),  v_{n} (s, x) \right), \quad (s, x) \in [0, T] \times \mathbb{R}^{\dimX}.
	\end{split}
\end{equation*}
Thus, it has already contained the information of the derivatives of the decoupling field. 
Therefore, if we consider the derivative with respect to the spatial variable, we need to compute the {\it second derivative} as $\nabla_{x}^2 X_{n}^{t, x}$.  It is generally hard since we have 
\begin{equation*}
	\nabla_x u_n=\left(1-\left(\nabla_x u_n \right)\left(\nabla_z \sigma\right)\right)^{-1}\left(\left(\nabla_x^2 u_n\right) \sigma+\nabla_x u_n \nabla_x \sigma+\left(\nabla_x u_n \right)^2 \nabla_y \sigma\right).
\end{equation*}
In this paper, we consider an alternative method as follows. 
\begin{Lem}\label{lem:equicontinuous vn}
Under \eqref{ass: smooth} and \eqref{ass: holder},  
\begin{equation*}
\begin{split}
	\sup_{n \in \mathbb{N}}  \sup_{x \in \mathbb{R}^{\dimX}}  \sup_{0 \leq t  \leq T} 
	|  v_n (t, x+h) - v_n (t, x)| 
	\leq C |h| ^{1/2}	. 
\end{split}
\end{equation*}

\end{Lem}
\begin{proof}
Again, as we have a dominated property formally provided by Lemma \ref{lem: Y dominated by X}, it implies that there exists a constant $C>0$ such that 
\begin{equation*}
	\begin{split}
	&\int_{t}^{T} \mathbb{E} [ \left| Z^{t, x+h}_n (s) -Z^{t, x}_n (s) \right|^2 ]	\mathrm{d}s
	\\&\leq C \, \mathbb{E} [ \left| \varphi (X^{t, x+h }_n (T)) -\varphi (X^{t, x }_n (T))\right|^2 ]
	\\&+ C \int_{t}^{T} \mathbb{E} \left[ \left| f (s, X^{t, x+h }_n (s), Y^{t, x }_n (s), Z^{t, x }_n (s))  - f (s, X^{t, x }_n (s), Y^{t, x }_n (s), Z^{t, x }_n (s)) \right|^2 \right]\mathrm{d}s.
	\end{split}
\end{equation*}
It follows from the same argument of Lemma \ref{lem:equicontinuous un} that 
\begin{equation*}
	\mathbb{E} \left[ \int_{t}^{T} \left|  Z^{t, x+h}_n (s) -Z^{t, x}_n (s)\right|^2 \mathrm{d}s\right]	 \leq \Lambda
	 \mathbb{E}[ \int_{t}^{T} |X^{t, x+ h  }_n (s) - X^{t, x  }_n (s)| \mathrm{d}s] .
\end{equation*}
Therefore, we have for all $r \in [t, T]$, we have 
\begin{equation*}
	\mathbb{E} \left[ \int_{t}^{r} \left|  Z^{t, x+h}_n (s) -Z^{t, x}_n (s) \right|^2 \mathrm{d}s\right]\leq C (r-t) |h|, \quad h \in \mathbb{R}^{\dimX}. 
\end{equation*}
Putting $r=t+ \epsilon$ for $\epsilon>0$, we have 
\begin{equation*}
\begin{split}
	\frac{1}{\epsilon}\mathbb{E}\left[ \int_{t}^{t+\epsilon}  |Z_{n}^{t, x+h} (s) -Z_{n}^{t, x} (s) |^2 \mathrm{d}s\right]
	\leq C |h| , \quad  h \in \mathbb{R}^{\dimX}	. 
\end{split}
\end{equation*}
Letting $\epsilon \to 0$, we obtain that 
\begin{equation*}
\begin{split}
	|v_{n}(t, x+h) -v_{n}(t, x)  |^2 = 
	\mathbb{E}\left[  |Z_{n}^{t, x+h} (t) -Z_{n}^{t, x} (t) |^2 \right]
	\leq C |h|  , \quad  h \in \mathbb{R}^{\dimX}	. 
\end{split}
\end{equation*}

\end{proof}

\subsection{Degenerate and uniformly linear growth}
In order to select a sub-sequence of the series $(u_n, v_n)$, 
we need a linear growth condition independent of $n \in \mathbb{N}$.
Then, we consider a {\it frozen } solution such that
for all $t \in [0,T]$, it holds that for some $x_0 \in \mathbb{R}^{\dimX}$, 
	\begin{equation*}
		\left( X_{n}^{t, x_0} (r), Y_{n}^{t, x_0} (r), Z_{n}^{t, x_0} (r) \right) \equiv \left( x_0,  \varphi (x_0), 0 \right), \quad (n, r) \in \mathbb{N} \times [t, T].
	\end{equation*}
In fact, if we have the frozen solution, it holds that for any $t \leq s \leq T$ it holds that 
\begin{equation*}
	\mathbb{E} [|X^{t, x}_{n}(s)|] 
	\leq \mathbb{E} [|X^{t, x}_{n}(s)-x_0|] + |\varphi (x_0)| \leq  
	|x-x_0| + |\varphi (x_0)|, 
\end{equation*}
where the last inequality is followed from an exponential inequality, Lemma \ref{lem: a key lemma: Lipschitz version}. 
In fact,  we provide the following lemma. 
\begin{Lem}\label{lem:linear growth un vn}
Suppose that  \eqref{ass: smooth} and \eqref{ass: degeneration} hold. Then, 
	\begin{equation*}
	\begin{split}
		\sup_{(n, t) \in \mathbb{N}\times[0,T]}\left\{ | u_{n}(t, x_0) | + | v_{n}(t, x_0) | \right\}
		<\infty , \quad x_0 \in \mathcal{Z}_0.
	\end{split}
	\end{equation*}	
Moreover, if we have \eqref{ass: holder}, we obtain 
\begin{equation*}
\begin{split}
	\sup_{n  \in \mathbb{N}}\sup_{0 \leq t \leq T}| u_n(t, x) | &\leq C(1+ |x|), 
	\\ \sup_{n \in \mathbb{N}}   \sup_{0 \leq t \leq T} 
	|  v_n (t, x)| &\leq C(1+ |x|) 	, \quad x \in \mathbb{R}^{\dimX}. 
\end{split}
\end{equation*}
\end{Lem}
\begin{proof}
We shall show that for all $t \in [0,T]$, it holds that 
	\begin{equation*}
		\left( X_{n}^{t, x_0} (r), Y_{n}^{t, x_0} (r), Z_{n}^{t, x_0} (r) \right) \equiv \left( x_0, \varphi (x_0), 0 \right) , \quad (n, r) \in \mathbb{N} \times [t, T].
	\end{equation*}
	By the definition it holds when $n=1$. 
	Suppose that it holds for any fixed $k \geq 1$. 
	Consider 
	\begin{equation*}
		X^{t, x}_{k+1} (r) = x + \int_{t}^{r} \sigma \left( s, X^{t, x}_{k+1} (s), \varphi ( X^{t, x}_{k+1} (s) ), 0 \right) \mathrm{dW}(s).
	\end{equation*}
	It follows from $x \in \mathcal{Z}_0$ that $X^{t, x_0}_{k+1} (r)  \equiv x_0$, and which implies $Z_{k+1}^{t, x_0} (r) \equiv 0$ for all $t \leq r \leq T$. 
	Moreover, $Y_{k+1}^{t, x}$ is a solution to the ordinary differential equation, 
\begin{equation*}
	Y^{t, x}_{k+1} (r) = \varphi \left( x \right) + \int_{r}^{T} f \left( s, x , Y^{t, x}_{k+1} (s), 0 \right) \mathrm{d}s, \quad r \in [t,T]. 
\end{equation*}
	In particular,  when $x = x_0$,  $Y_{k+1}^{t, x} (r) = \varphi (x_0)$ holds for $t \leq r \leq T$. Thus, we conclude for any $t \geq 0$ and $ (n, r) \in \mathbb{N}\times[t,T]$,
	\begin{equation*}
	\begin{split}
		Y^{t, x_0}_{n} (r) = u_{n}(r, x_0) = \varphi (x_0) , \quad 
		Z^{t, x_0}_{n} (r) = v_{n}(r, x_0) =0.
	\end{split}
	\end{equation*}
	It follows that 
	\begin{equation*}
	\begin{split}
		\sup_{(n, t) \in \mathbb{N}\times[0,T]}\left\{ | u_{n}(t, x_0) | + | v_{n}(t, x_0) | \right\}
		<\infty  .
	\end{split}
	\end{equation*}	
\end{proof}

\subsection{Markov property}
The forward process is given by the diffusion process, 
thus it satisfies a Markov property. This is convenient to study the continuity with respect to time for the decoupling field.
\begin{Lem}
\label{lem:equicontinuous un w.r.t. time}
Under \eqref{ass: smooth}, \eqref{ass: holder} and \eqref{ass: degeneration},  it holds that 
\begin{equation*}
		\sup_{(n, x) \in \mathbb{N}\times\mathbb{R}^{\dimX}}  \sup_{0 \leq t  \leq T}\mathbb{E} |u_{n}(t+h, x) -u_{n}(t, x)  | \leq  C |h|^{1/2} . 
\end{equation*} 
\end{Lem}
\begin{proof}
For any $ (n, t, x) \in \mathbb{N}\times[0,T]\times \mathbb{R}^{\dimX}$, we consider 
\begin{equation*}
	\begin{split}
	&u_{n}(t+h, x) - u_{n}(t, x) 
	\\&= \mathbb{E} \left[
	\varphi (X_{n}^{t+h, x}(T)) - \varphi (X_{n}^{t, x}(T))  + 
	\int_{t+h}^{T} f(s, \Theta_n^{t+h, x} (s)) \mathrm{d}s 
	-\int_{t}^{T} f(s, \Theta_n^{t, x} (s)) \mathrm{d}s 
	\right]
	\\&= \mathbb{E} \left[
	\varphi (X_{n}^{t+h, x}(T)) - \varphi (X_{n}^{t, x}(T))  + 
	\int_{t+h}^{T} f(s, \Theta_n^{t+h, x} (s)) - f(s, \Theta_n^{t, x} (s)) \mathrm{d}s 
	\right]
	\\&- \mathbb{E} \left[
	\int_{t}^{t+h} f(s, \Theta_n^{t, x} (s)) - f(s, \Theta_n^{t, x_0} (s)) \mathrm{d}s 
	+\int_{t}^{t+h} f(s, \Theta_n^{t, x_0} (s))  \mathrm{d}s 
	\right]. 
	\end{split}
\end{equation*}
For the last integral term, we note that 
\begin{equation*}
	\mathbb{E} \left[
	\int_{t}^{t+h} f(s, \Theta_n^{t, x_0} (s))  \mathrm{d}s 
	\right]
	= u_{n}(t, x_0)- \mathbb{E}[\varphi \left( X^{t, x_0}(t+h) \right) ]=0.
\end{equation*}
Therefore, it follows from the same argument that there exists $C>0$ such that 
for all $ t+h \leq r \leq T$, 
\begin{equation*}
\begin{split}
	& | u_{n}(t+h, x) - u_{n}(t, x)|^2  
	\\ &\leq 
	C \left\{ \sup_{0 \leq r \leq T} \mathbb{E} [ |X^{t+h, x}_n (r)- X^{t, x}_n(r)| ]+
		h \sup_{0 \leq r \leq T} \mathbb{E} [ |X^{t, x}_n (r)- X^{t, x_0}_n(r)| ] 
	\right\}.	
\end{split}
\end{equation*}
By the Markov property, it holds that for $t+h \leq r \leq T$, 
\begin{equation*}
\begin{split}
	& X_n^{t, x}(r)  =  X_n^{t+h , y}(r),  \quad 
	\quad  y = X_n^{t, x } (t+h). 
\end{split}
\end{equation*}
Moreover, it follows from Lemma \ref{lem: a key lemma: Lipschitz version} that 
\begin{equation*}
	\begin{split}
		&\mathbb{E} \left[ |X^{t+h, x}_n (r) - X^{t, x}_n (r) | \right]
		\\&=\mathbb{E} \left[ \mathbb{E}\left[|X^{t+h, x}_n (r) - X^{t+h, y}_n (r) |\right] |_{y = X_n^{t, x } (t+h)} \right]
		\\&\leq  \sqrt{l} \mathbb{E} \left[ | x - X_n^{t, x } (t+h)  | \right]. 
	\end{split}
\end{equation*}
It follows the linear growth property of $\sigma$ from  \eqref{ass: smooth} that for any $t \leq r \leq t+h$,  we have 
\begin{equation*}
	\begin{split}
		\mathbb{E} \left[ | X_n^{t, 0 } (r)  |^2  \right]
		&\leq 
		\mathbb{E} \left[ \int_{t}^{r} | \sigma (s, X^{t, 0}_n (s), u_{n-1}(s, X^{t, 0}_n (s)), v_{n-1}(s, X^{t, 0}_n (s)) |^2 \mathrm{d}s \right] 
		\\ &\leq C|h| + 
		C \mathbb{E} \left[ \int_{t}^{r} | X^{t, 0}_n (s)|^2 + | u_{n-1}(s, X^{t, 0}_n (s))|^2 + | v_{n-1}(s, X^{t, 0}_n (s))|^2  \mathrm{d}s \right] .
	\end{split}
\end{equation*}
Notify that $u_n$ and $v_n$ are uniformly linear growth with respect to $n$. Thus, we have a constant $C>0$ such that  
\begin{equation*}
	\begin{split}
		& \mathbb{E} \left[\int_{t}^{r}  | u_{n-1}(s, X^{t, 0}_n (s))|^2 + | v_{n-1}(s, X^{t, 0}_n (s))|^2  \mathrm{d}s  \right] 
		\\ &\leq C\left\{ |h| + 
		 \mathbb{E} \left[ \int_{t}^{r} | X^{t, 0}_n (s)|^2 \mathrm{d}s\right]  \right\}.
	\end{split}
\end{equation*}
It follows from the Gronwall inequality that there exists $C>0$ such that 
\begin{equation*}
	\mathbb{E} \left[ | X_n^{t, 0 } (t+h)  |^2  \right] \leq C |h|.
\end{equation*}
Thus, we concludes that 
\begin{equation*}
	\sup_{ 0 \leq t \leq T}|u_n(t+h, x) -u_n(t, x)|  \leq C |h|^{1/2}.
\end{equation*}
\end{proof}

\begin{Rem}[Uniformly bounded Lipschitz continuous]\label{rem:Uniformly bounded Lipschitz continuous}
	We note that Lipschitz continuous and uniformly bounded functions for spatial forward variable $x$ satisfies \eqref{ass: holder} since 
	\begin{equation*}
		| \varphi (x+h) - \varphi (x) | \leq 
		\begin{cases}
			L_{\varphi, x} |h| \leq L_{\varphi, x} |h|^{1/2}, 			\quad &|h| \leq 1,\\
			2 \| \varphi \|_{\infty} \leq 2 \| \varphi \|_{\infty} |h|^{1/2}, \quad &|h| >1 .
		\end{cases}
	\end{equation*}
	where $\| \cdot  \|_{\infty}$ stands for the sup-norm; $\| \varphi \|_{\infty} := \sup_{x \in \mathbb{R}^{\dimX} } | \varphi  (x) |$. 
\end{Rem}

\section{Global solutions}\label{sec:Global solutions}
Gathering the above results, we shall show the existence and uniqueness of the global solution to 
the drift-less forward-backward SDEs \eqref{drifted-less FBSDEs}. 
\begin{Thm}
Suppose that Assumption \eqref{ass: all condition} holds.  
Then, there exists a solution to the FBSDE \eqref{drifted-less FBSDEs}. 
Moreover, 
when $\dimX =1$, it is a unique solution. 
\end{Thm}

\begin{proof}
It follows from Lemma \ref{lem:equicontinuous un} and \ref{lem:equicontinuous vn} that 
the continuous functions $\{ (u_{n}, v_{n}) \}_{n \in \mathbb{N}}$ 
satisfies uniformly equicontinuous and uniformly bounded. 
Applying Cantor's diagonal argument, we obtain that  $\{ (u_{n}, v_{n}) \}_{n \in \mathbb{N}}$ has a subsequence such that there exists a pair of continuous functions $(u, v)$ such that 
\begin{equation*}
	\lim_{k \to \infty} \delta_k := 
	\lim_{k \to \infty}| (u_{n(k)}, v_{n(k)}) (s, x) - (u, v) (s, x) |=0 , \quad (s, x) \in [0,T] \times \mathbb{R}^{\dimX}, 
\end{equation*}
where we denote $\delta_k =| (u_{n(k)}, v_{n(k)}) (s, x) - (u, v) (s, x) |$ for convenient. 
We note that as the closed interval $[0,T]$ is compact set, 
the above convergence holds uniformly with respect to time variable. 

Since $ (u_{{n(k)}}, v_{{n(k)}}) $ is 
satisfies Lipschitz continuous for all $k \in \mathbb{N}$ and 
$(u, v)$ is a continuous function, 
then there exists the corresponding weak solutions such that 
\begin{equation*}
\begin{split}
		X_{n(k)}(r) &= x 
		+ \int_{t}^{r} {\sigma}_k \left( s, X_{n(k)} (s)  \right)\mathrm{dW}(s), \quad t \leq r \leq T, 
		\\X(r) &= x + \int_{t}^{r} \sigma \left( s, X (s)  \right)\mathrm{dW}(s), \quad t \leq r \leq T, 
\end{split}
\end{equation*}
where we denote $\sigma_k (x) = \sigma (s, x, u_{n(k)-1}(s, x), v_{n(k)}(s, x) )$ and $\sigma (x) = \sigma (s, x, u(s, x), v(s, x) )$ for $(s, x, k) \in [0,T] \times \mathbb{R}^{\dimX} \times \mathbb{N}$. 
Note that the weak solution exists if the diffusion coefficient is continuous, see \cite{skorokhod1965studies}. 

Next, we shall show that the series $\{ X_{n(k)} \}$ converges to $X$ in $\mathscr{S}^2$. 
It follows from Lemma \ref{lem:equicontinuous un} and the point-wise convergence that it holds for all $(s, x, h)\in [0,T] \times \mathbb{R}^{\dimX}\times \mathbb{R}^{\dimX} $, 
\begin{equation*}
\begin{split}
	&|  u (s, x+h) - u (t, x)| \leq C |h| ^{1/2}, 
	\\&|  v (s, x+h) - v (t, x)| \leq C |h| ^{1/2}	.
\end{split}
\end{equation*}
Therefore, there exists a constant $C>0$ such that for all $s \in [0,T]$, $x_1, x_2 \in \mathbb{R}^{\dimX}$ and $k \in \mathbb{N}$ it holds that 
\begin{equation*}
	\begin{split}
		\left| {\sigma}_k \left( s, x_1  \right)-\sigma \left( s, x_2 \right) \right|^2 
		\leq C \left( |x_1 -x_2|^2 + |x_1 -x_2| +\delta_k^2 \right).
	\end{split}
\end{equation*}
It leads that there exists a constant $C>0$ such that 
for all $t \leq r \leq T$ we have 
\begin{equation*}
	\begin{split}
		\left| X_{n(k)}(r) - X(r)\right|^2 
		\leq C \int_{t}^{r} \left( |X_{n(k)}(s) - X(s)|^2 + |X_{n(k)}(s) - X(s)| +\delta_k^2 \right) \mathrm{d}s. 
	\end{split}
\end{equation*}
Denoting $g_{k}(s) = \mathbb{E}[\sup_{ t \leq u \leq s } \left| X_{n(k)}(u) - X(u) \right|^2]$ for $s \in [t,T]$ and Jensen's inequality, we obtain 
\begin{equation*}
	\begin{split}
		g_{k}(r)
		\leq C \int_{t}^{r} \left( g_{k}(s) + \sqrt{g_{k}(s)}+\delta_k^2 \right) \mathrm{d}s , \quad t\leq r \leq T, \ k \in \mathbb{N}. 
	\end{split}
\end{equation*}
For all $ L>0$, we have that 
\begin{equation*}
	\begin{split}
		g_{k}(r) e^{-C L r}
		\leq C \int_{t}^{r} \left( (1-L) g_{k}(s) + \sqrt{g_{k}(s)}+\delta_k^2 \right) e^{-s L}\mathrm{d}s , \quad t\leq r \leq T, \ k \in \mathbb{N}. 
	\end{split}
\end{equation*}
Thus, putting  $L>1$, we obtain a constant $C>0$ such that  for all $t\leq r \leq T, \ k \in \mathbb{N}$, 
\begin{equation*}
	\begin{split}
		g_{k}(r) 
		&\leq C \int_{t}^{r} \left( \sqrt{g_{k}(s)}+\delta_k^2 \right) \mathrm{d}s .
	\end{split}
\end{equation*}
Let us consider the series  of function $\{g_k (r)\}_{k \in \mathbb{N}}$. It follows from Lemma \ref{lem:linear growth un vn} that for all $r \in [t,T]$ and $k \in \mathbb{N}$,  
\begin{equation*}
	  g_{k}(r) \leq  \sup_{k \in \mathbb{N}}\mathbb{E} \left[ \sup_{ t \leq r \leq T } \left| X_{n(k)}(r) \right|^2 \right] < \infty. 
\end{equation*}
Therefore, applying the reverse Fatou's lemma and putting $g(r)=\limsup_{k \to \infty} g_k (r)$ for $r \in [t,T]$ and we obtain the following inequality, 
\begin{equation*}
	\begin{split}
		g(r) 
		&\leq C  \int_{t}^{r} \sqrt{g(s)} \mathrm{d}s , \quad r \in [t,T].
	\end{split}
\end{equation*}
Thanks to the Osgood's criterion that we have $g \equiv 0$. 
Therefore, we conclude 
\begin{equation*}
	X_{n(k)}  \xrightarrow[\| \cdot \|_{\mathscr{S}^2 } ]{k \to \infty}  X.
\end{equation*}
Furthermore, we obtain the convergence of the backward process: 
\begin{equation*}
	(Y_{n(k)}, Z_{n(k)}) = \left( u_{k} (\cdot, X_{n(k)} ), v_{n(k)} (\cdot, X_{n(k)} ) \right) \xrightarrow[\| \cdot \|_{\mathscr{S}^2 \times \mathscr{H}^2 } ]{k \to \infty}  (Y, Z) =  \left( u(\cdot, X ), v (\cdot, X ) \right) .
\end{equation*}
As a map $s \mapsto v_{n(k)}(s, x)$ for $x \in \mathbb{R}^{\dimX}$ may be discontinuous, 
the set of the discontinuity points in $v_{n(k)}(\cdot, x)$ is at most countable. Thus, we have 
\begin{equation*}
	\mathbb{P}\left(\lim _{k \rightarrow \infty} v_{n(k)}(s, X(s)) =v(s, X(s)) \right)=1, \quad t \leq s \leq T. 
\end{equation*}
Then, the pair $(X, Y, Z)$ satisfies the system of FBSDE \eqref{drifted-less FBSDEs}. Finally, we obtain the existence of the solution of the FBSDE.  

In particular, for one-dimensional case, 
the uniqueness is followed from Yamada-Watanabe theorem \cite{YW} that of the forward SDEs's uniqueness with the $(1/2)$-H\"older continuous diffusion coefficients.  
\end{proof}

\begin{Rem}[Existence for ODEs and SDEs]\label{lem:Existence for ODEs and SDEs}
	Under the Lipschitz condition, the existence follows on the forward drift-less SDEs. 
	For ODEs, there is an example such that no global solution exists: 
	\begin{equation}\label{ex: Fromm's example}
		\begin{dcases}
		 X(r) &=x  +    \int_{t}^{r} L_{b, y} Y (s)  \mathrm{d}s  + \int_{t}^{r} L_{\sigma, y }Y(s)  \mathrm{dW}(s)
	\\   Y(r) &=  X (T) - \int_{r}^{T} Z (s) \mathrm{dW}(s), \quad  r \in [t , T ] .
		\end{dcases}
	\end{equation}
	
	If $L_{b, y}  =1$, $L_{\sigma, y }=0$ and $T>1$ hold, then 
	it has a local solution such that it can not be the global solution since 
	\begin{equation*}
		 (X, Y, Z)(r) = \left(x+ \frac{x(r-t)}{1- (T-t)} ,  \frac{x}{1- (T-t)},  0 \right) , \quad t \in (T-1,T]. 
	\end{equation*}	
	The system \eqref{ex: Fromm's example} is equivalent to 
	\begin{equation*}
		\begin{dcases}
		 X(r) &=x  +  \int_{t}^{r} Z (s) \mathrm{dW}(s) ,
	\\   Y(r) &=  X (T) -  \int_{r}^{T} Y (s)  \mathrm{d}s  - \int_{r}^{T} Z (s) \mathrm{dW}(s), \quad  r \in [t , T ] .
		\end{dcases}		
	\end{equation*}
	Therefore, for the above example, the unique and existence admits if and only if the initial values start only from zero; $x = 0\Longleftrightarrow \mathcal{Z}_0 \neq \emptyset$.
	
	If $L_{b, y}  =0$ and $L_{\sigma, y }=1$ hold, it is  a degenerate case;
	\begin{equation*}
		\begin{dcases}
		 X(t) &= x +  \int_{t}^{r} Y(s)  \mathrm{dW}(s), 
	\\   Y(t) &= X (T) - \int_{r}^{T} Z (s) \mathrm{dW}(s), \quad  t \in[0, T]. 
		\end{dcases}
	\end{equation*}
	Our result shows that it has a solution, in fact it has the following expression, 
	\begin{equation*}
		 X(t) = Y(t) = Z(t) = x \exp \left(\frac{1}{2}t -W(t) \right), \quad u(t, x)=x. 
	\end{equation*}
\end{Rem}

\begin{Rem}[Comparing the condition: $ L_{\sigma,z}L_{\varphi,x}  \neq 1 $]	
	The uniqueness is a different problem for the diffusion non-degenerate coefficients. 
	In fact, 
	for arbitrary progressive measurable $\sigma_0$, let us consider one dimensional case such that 
	\begin{equation*}
	\begin{dcases}
	X ( t )  &=   X  ( 0 ) + (1- L_{\sigma,z}L_{\varphi,x} )^{-1}\int _ { 0 } ^ { t } \sigma_0 ( s , \omega )  \mathrm{dW}(s)\\
	Y ( t )  &= L_{\varphi,x} \left\{ X  ( 0 ) + (1- L_{\sigma,z}L_{\varphi,x} )^{-1}\int _ { 0 } ^ { t } \sigma_0 ( s , \omega )  \mathrm{dW}(s)\right\} , \\
	Z ( t )  &= L_{\varphi,x} (1- L_{\sigma,z}L_{\varphi,x} )^{-1} \sigma_0 ( t , \omega ),  \quad t \in [0,T].  
	\end{dcases} 
	\end{equation*} 
	Thus, the above $(X, Y, Z)$ solves the following FBSDE; 
	\begin{equation*}
	\begin{dcases}
	X  ( t )  &= X  ( 0 ) + \int _ { 0 } ^ { t } [ L_{\sigma,z}  Z ( s ) + \sigma_0 ( s , \omega ) ] \mathrm{dW}(s), \\
	Y  ( t )  &= L_{\varphi,x} X ( T ) (\omega) - \int _ { t } ^ { T } Z ( s )  \mathrm{dW}(s). 
	\end{dcases} 
	\end{equation*}
	Thus, this is well defined if and only if $ L_{\sigma,z}L_{\varphi,x}  \neq 1 $ for non-degenerate $\sigma_0$.  
	For degenerate case; $\mathcal{Z}_0 \neq \emptyset \Leftrightarrow \sigma_0 \equiv 0$, 
	 the existence and uniqueness hold even if  $ L_{\sigma,z}L_{\varphi,x}  = 1 $. 
\end{Rem}

\section{Appendix}\label{sec:Martingale estimation}
In this section, 
for a given $\sigma$, 
and for any $(t, x) \in [0,T] \times \mathbb{R}^{\dimX}$,  let $X^{t, x} = \left\{ X^{t, x} (r) \right\}_{r \in [t,T]}$ be a solution to the equation, 
	\begin{equation}\label{eq:drift-less forward SDE}
		X (r) = x + \int_{t}^{r} \sigma \left(s, X (s)\right) \mathrm{dW}(s), \quad t \leq r \leq T. 
	\end{equation}
\begin{Lem}[Key lemma]\label{lem: a key lemma: Lipschitz version}
Suppose that $\sigma$ satisfy \eqref{ass: smooth}. 
Let $X^{t, x} = \left\{ X^{t, x} (r) \right\}_{r \in [t,T]}$ be a solution to the equation \eqref{eq:drift-less forward SDE}. 
	Then, for all $0 \leq t \leq r$ and $x , h \in  \mathbb{R}^{\dimX}$, 
\begin{equation*}
	 \mathbb{E} \left[\left|  X^{t, x+ h  } (r) - X^{t, x} (r) \right|\right]
	\leq	\sqrt{\dimX} \left|  h\right|   . 
\end{equation*}
\end{Lem}

\begin{proof}
	For any fixed $(t, x) \in [0,T] \times \mathbb{R}^{\dimX}$, we write 
	for $j=1, 2, \dots, \dimX$, 
	$e_{(j)} = (0, \dots, 0, \overbrace{1}^{j}, 0, \dots, 0)$ and 	we set for any fixed $\epsilon \in (-1,1) \setminus \{0\}$, $\delta \geq 0$ and  $j=1, 2, \dots, \dimX$, 
	\begin{equation*}
		  M^{\epsilon^{-1} \delta e_{(j)} } (r) \equiv \epsilon^{-1} \left(  X^{t, x+  \delta e_{(j)} } (r) - X^{t, x} (r) \right), \quad r \in [t,T]. 
	\end{equation*}
	Thus, we have 
	\begin{equation*}
		\begin{split}
		&M^{\epsilon^{-1} \delta e_{(j)} } (r) = \epsilon^{-1}\delta e_{(j)} 
		\\&+ \sum_{k=1}^{\dimB}\int_{t}^{r} 
		\Sigma_{k, j}\left(s,  X^{t, x+  \delta e_{(j)} } (s),  X^{t, x} (s) \right) 
		   M^{\epsilon^{-1} \delta e_{(j)} } \mathrm{d}W_{k} (s) 
		\end{split}
	\end{equation*}
	where $\Sigma_{k, j}$ is a function such that for $k=1, \dots, \dimB$ and $j=1, \dots, \dimX$
	\begin{equation*}
		\Sigma_{k, j} (s, a, b) = \int_{0}^{1} (\nabla_x \sigma)_k \left(s, (1-\theta) a  + \theta b \right) \mathrm{d} \theta,  \quad (s, a, b) \in [0,T] \times \mathbb{R}^{2\dimX} , 
	\end{equation*}	
	where $(\nabla_x \sigma)_k$ stands for the $k$-component of the space derivative $\sigma$. 
	We note that 
	\begin{equation*}
		\left\{ \left( X^{t, x+ \delta e_{(j)} } (r) ,  X^{t, x} (r), M^{\epsilon^{-1} \delta e_{(j)} } (r) \right) \right\}_{r \in [t,T]}
	\end{equation*}
	is a solution to the SDE with Lipschitz continuous coefficient. 
	In particular, when $x=\delta=0$ it has the trivial solution, $0 \in \mathbb{R}^{3\dimX}$. 
	It follows from the pathwise uniqueness that 
	the $\dimX$-dimensional process $M^{\epsilon^{-1} \delta e_{(j)} }$ satisfies  
	\begin{equation*}
	M^{\epsilon^{-1} \delta e_{(j)} }  = (0, \dots, 0,\epsilon^{-1} \delta \mathcal{E}(X, j), 0, \dots, 0),
	\end{equation*}
	where we denote 
	\begin{equation*}	
	\mathcal{E}(X, j) =
		  \exp \left[ -\frac{1}{2}  \int_{t}^{r} \langle \Sigma_{\cdot, j}(s),  \Sigma_{\cdot, j}(s) \rangle_{\mathbb{R}^{\dimB}}\mathrm{d}s+ \sum_{k=1}^{\dimB}\int_{t}^{r}  \Sigma_{k, j}(s) \mathrm{d}W_{k} (s)\right]. 
	\end{equation*}	
	In short, we conclude that 
	for any $x \in \mathbb{R}^{\dimX}$ and $\delta \geq 0$, it holds 
	\begin{equation*}
		 \mathbb{E}[|  X^{t, x+  \delta e_{(j)} } (r) - X^{t, x} (r)  |]
		 = \mathbb{E}[|  \delta \mathcal{E}(X, j)   |]
		= \left|  \delta \right|, \quad j=1,2, \dots, \dimX. 	
	\end{equation*}
	Moreover, for any $h=(h_{1}, h_{2}, \dots, h_{\dimX}) \in \mathbb{R}^{\dimX}$, 
	it holds that 
	\begin{equation*}
		X^{x+h}(r) -X^{x}(r) = 
		(h_{1} \mathcal{E}(X, 1), h_{2} \mathcal{E}(X, 2), \dots,  h_{\dimX} \mathcal{E}(X, \dimX)).
	\end{equation*}
	It leads that for any $(t, x) \in [0,T] \times \mathbb{R}^{\dimX}$, we have
	\begin{equation*}
		\begin{split}
			&\mathbb{E} \left[ \left|  X^{t, x+h}(r) -X^{t,x}(r) \right| \right] 
			=  \mathbb{E} \left[ \left\{\sum_{j=1}^{\dimX}|h_{j}|^2 |\mathcal{E}(X, j) (r)|^2 \right\}^{\frac{1}{2}} \right] 
		\\&\leq \mathbb{E} \left[ \sum_{j=1}^{\dimX} | h_j | | \mathcal{E}(X, j)(r)| \right]
		= \sum_{j=1}^{\dimX} | h_j | \mathbb{E} [ |\mathcal{E}(X, j)(r)|]
		=\sum_{j=1}^{\dimX} | h_j |.
		\end{split}
	\end{equation*}
	The desired result is followed from the inequality, 
	\begin{equation*}
	\begin{split}
		\sum_{j=1}^{\dimX} | h_j |
		\leq \sqrt{\dimX} \left( \sum_{j=1}^{\dimX} | h_j |^2 \right)^{1/2} .
	\end{split}
	\end{equation*}

\end{proof}

\begin{Rem}
	This estimate holds for the $L^1$ but not $L^2$. 
	In fact, we consider an exponential martingale; 
	\begin{equation*}
		X(r) = x + \int_{t}^{r} L_{\sigma, x}\, X(s)\mathrm{dW}(s)
		= x \exp \left( {L_{\sigma, x}} W(r) -\frac{1}{2}{L_{\sigma, x}^2} r \right). 
	\end{equation*} 
	Thus we have 
	\begin{equation*}
		\mathbb{E}[|X^{t, x+h}(r) - X^{t, x}(r)|^p] =h^p \exp \left( \frac{1}{2}\left(p^2 - 1\right){L_{\sigma, x}^2} r \right), \quad p>0. 
	\end{equation*}
	From the observation, the above inequality is interesting itself. 
\end{Rem}

\begin{Rem}[Calder\'on Zygmund lemma]
A key of the estimation is that of heat kernel, thus, which can be proven by a fundamental analytical approach known as Calder\'on Zygmund lemma. 
Moreover, the transition function of SDE, has a uniformly equicontinuous for the time and space if the coefficients satisfies a {\it bounded} condition see the book of \cite[Theorem 7.2.4]{Stroock1997}.
For one dimensional unbounded setting, a simple proof also given by \cite{Tsuchiya2021v2}. 
\end{Rem}

\subsection{Dominated property and decoupling FBSDEs}
We note a relation between the decoupling expression and the non-linear Feynman-Kac formula. 
According to the result was given \cite{PardouxPeng1992}, the decoupling FBSDEs; 
\begin{equation*}	
	\begin{dcases}
		X(r) = X (t) + \int_{t}^{r} b \left(s, X(s) \right) \mathrm{d}s+ \int_{t}^{r} \sigma \left(s, X(s) \right) \mathrm{dW}(s)
		\\ Y(r) = \varphi \left( X(T) \right) + \int_{r}^{T} f \left(s, X(s), Y(s), Z(s) \right) \mathrm{d}s - \int_{r}^{T} Z (s) \mathrm{dW}(s), \quad r \in [t,T]. 
	\end{dcases}
\end{equation*}
gives a expression such that there exists a pair of functions, $(u, v)$ satisfies 
\begin{equation*}
		Y(s) = u(s, X(s)), \quad Z(s) = \nabla_x u (s, X(s)) \cdot \sigma (s, X(s) ), \quad s  \in [0,T]. 
\end{equation*}
Therefore, we obtain 
\begin{Lem}\label{lem: an expression via FK and BSDE uniqueness}
Suppose that \eqref{ass: smooth} holds. 
Then, there exists a series $(u_n, v_n)$ satisfying 
the equation \eqref{decoupling FBSDE approximation via u and v} and 
it is uniquely determined for all $n \in \mathbb{N}$. 
\end{Lem}
\begin{proof} 
Firstly, 
for $(t, x) \in [0,T]\times \mathbb{R}^{\dimX}$, we denote $X_{0}^{t, x} (r) \equiv x$ for all $t \leq r \leq T$. 
Thus, the \eqref{ass: smooth} implies that $(Y_0, Z_0) \in \mathscr{S}^2 \times \mathscr{H}^2$ be a unique solution of BSDE such that 
\begin{equation*}
	Y_{0} (r) = \varphi \left( X_{0} (T) \right) + \int_{r}^{T} f \left( s, X_{0}(s), Y_{0}(s), Z_{0}(s) \right) \mathrm{d}s-   \int_{r}^{T}  Z_{0}(s) \mathrm{d}s, \quad  t \leq r \leq T. 
\end{equation*}
Moreover, we have for all $(t, x) \in [0,T] \times \mathbb{R}^{\dimX}$, 
\begin{equation*}
Y_0 (r) =  \varphi ( x ) , \quad Z_0 (r)=0, \quad t \leq r \leq T . 
\end{equation*}
Denoting $ u_{0}(t, x)= \varphi ( x )$ and $v_{0}(t, x)=0$ for all $(t, x) \in [0,T] \times \mathbb{R}^{\dimX}$, 
it follows from \eqref{ass: smooth} again that there exists a unique strong solution to the following SDE, 
\begin{equation*}
		X_{1} (r) = x + \int_{t}^{r} \sigma \left(s, u_0 (s, X_{1} (s)), v_0 (s, X_{1} (s) ) \right)  \mathrm{dW}(s), \quad t \leq r \leq T. 
\end{equation*}

For $k \in \mathbb{N}$, suppose that there exists a series of a pair of smooth functions $(u_{k-1}, v_{k-1})$ such that 
\begin{equation*}
	 v_{k-1} (t, x) = \left(\nabla_x u_{k-1}\right)(t, x) \cdot \sigma (t, x, u_{k-1}(t, x), v_{k-1}(t, x)), \quad (t, x) \in [0,T]\times \mathbb{R}^{\dimX}. 
\end{equation*}
It induces the following decoupling FBSDE's solution $\Theta_{k}^{t, x} $; 
\begin{equation*}
	\begin{dcases}
	X_{n(k)} (r) &= x + \int_{t}^{r} \sigma \left( s, X_{n(k)}(s), u_{k-1} (s, X^{t, x}_{k} (s) ), v_{k-1} (s, X^{t, x}_{k} (s) ) \right) \mathrm{dW}(s)
	\\ Y_{k} (r) &= \varphi \left( X_{n(k)} (T) \right) + \int_{r}^{T} f \left( s, \Theta_{k}^{t, x}(s) \right) \mathrm{d}s
	 -  \int_{r}^{T}  Z_{k}(s) \mathrm{dW}(s), \quad (t, x) \in [0,T] \times \mathbb{R}^{\dimX}. 
	\end{dcases}
\end{equation*}
Thanks to \cite{PardouxPeng1992}, there exists a semi-linear parabolic PDE solution $u_{k}$ such that it satisfies for all $t \leq r \leq T$, 
\begin{equation*}
	Y_{k}(r) = u_{k} (r, X_{n} (r) ), \ 
	Z_{k}(r) = \left( \nabla_x u_{k} \right)(r, X_{n(k)} (r) ) \cdot \sigma \left( \nabla_x u_{k} \right) (r, \Theta_{k}^{t, x}(r) ). 
\end{equation*}
In particular, putting $v_{n}(t, x) \triangleq Z^{t, x}_{n} (t)$ we have 
\begin{equation*}
	v_{n}(t, x) = \left( \nabla_x u_{n} \right)(t, x) \cdot  \sigma \left( t, x , u_{n}(t, x), v_{n}(t, x) \right), \quad (t, x) \in [0,T] \times \mathbb{R}^{\dimX}. 
\end{equation*}
The smooth condition of the diffusion coefficient $\sigma (s, u_{k} (s, x), v_{k} (s, x))$ implies that there exists a strong unique solution in $\mathscr{S}^2$ such that 
\begin{equation*}
	X_{k+1} (r) = x + \int_{t}^{r} \sigma \left(s, X_{k+1} (s), u_{k}(s, X_{k+1} (s) ), v_{k}(s, X_{k+1} (s) ) \right)\mathrm{dW}(s). 
\end{equation*}
Therefore, we obtain the desired result. 
	\end{proof}
	
	\begin{Rem}
	The series of functions $\{ u_n \}$ and $\{ v_n \}$ is Lipschitz continuous for all $(t, n) \in [0,T] \times \mathbb{N}$: 
	\begin{equation*}
	\begin{split}
		&\left| u_{n}(t, x+ h) - u_{n}(t, x) \right|	 \leq L_{u_n, x} |h|, 
		\\ &\left| v_{n}(t, x+ h) - v_{n}(t, x) \right|	 \leq L_{v_n, x} |h|, \quad (x, h) \in \mathbb{R}^{\dimX} \times \mathbb{R}^{\dimX}.  
	\end{split}
	\end{equation*}
	It stands for the construction of the scheme is easy but the problem is to select the convergence sequence if it exists. 
	\end{Rem}

	The following dominated property is convenient. 
\begin{Lem}[$\overline{Y}$ dominated by $\overline{X}$]\label{lem: Y dominated by X}
	For any $X_i \in \mathscr{S}^2 , \ (i=1,2)$, 
	$(Y_i, Z_i)$ be a solution of BSDEs, 
	\begin{equation*}
	\begin{split}
		&Y_i (r)=\varphi (X_i (T))  + \int_{r}^{T} f \left(  s, X_i (s), Y_i (s), Z_i (s) \right) \mathrm{d}s-  \int_{r}^{T} Z_i (s) \mathrm{dW}(s), \quad r \in [t,T].
	\end{split}
	\end{equation*}
	Denote and write 
	\begin{equation*}
		(\overline{X}, \overline{Y}, \overline{Z}) \triangleq (X_1 - X_2, Y_1 - Y_2, Z_1 - Z_2 ). 
	\end{equation*}
	For any $\epsilon \in (0,1)$ we denote 
	$L = (1+ L_{f, y}^2 +\epsilon^{-1} ) + (1-\epsilon ) $. 
	Then, it holds that 
	\begin{equation*}
	\begin{split}
		&	\mathbb{E} \left[ \left| \overline{Y} (t) \right|^2 e^{-Lt} \right] 
			+(1-\epsilon )\int_t^T \mathbb{E} \left[\left(  \left| \overline{Y} (s) \right|^2+  \left| \overline{Z} (s) \right|^2 \right)  e^{-Ls}\right] \mathrm{d}s
		 \\&\leq   \mathbb{E} \left[ \left| \delta_{X} \varphi (T)\right|^2 e^{-LT}\right]  +  \int_{t}^{T}    \mathbb{E} \left[ \left|\delta_{X} f (s) \right|^2 e^{-Ls} \right]  \mathrm{d}s. 
	\end{split}
	\end{equation*}
	where we define 
	$\delta_{X}  f (s) = f (s, X_1(s), Y_1 (s), Z_1 (s) ) - f (s, X_2 (s), Y_1 (s), Z_1 (s) )$ and 
	$ \delta_{X} \varphi (T)\triangleq \varphi \left( X_1 (T) \right)-\varphi \left( X_2 (T) \right)$.	
\end{Lem}
	\begin{proof}
	Applying the It\^o formula to $\left| \overline{Y} (t) \right|^2 e^{- L t  }$ for some $L>0$, we have 
	\begin{equation*}
	\begin{split}
			& \left| \overline{Y} (t) \right|^2 e^{-Lt} 
			+ L \int_t^T \left| \overline{Y} (s) \right|^2 e^{-Ls} \mathrm{d}s
			+  \int_t^T \left| \overline{Z} (s) \right|^2 e^{-Ls} \mathrm{d}s
		\\ &=   \left| \delta_{X} \varphi (T) \right|^2 e^{-LT} +  2 \int_{t}^{T}  \langle \overline{Y}(s), \overline{f} (s) \rangle e^{-Ls} \mathrm{d}s
		-   2 \int_{t}^{T} e^{-Ls} \langle \overline{Y}(s), \overline{Z}(s) \mathrm{d}W(s) \rangle
		.
	\end{split}
	\end{equation*}
	where we define $\overline{f} (s) = f (s, X_1(s), Y_1(s), Z_1 (s) ) - f (s, X_2 (s), Y_2 (s), Z_2 (s) )$. 
	For arbitrary $\epsilon>0$, we have $2 ab \leq {\epsilon a}^2 + {\epsilon^{-1} b^2}$ for all $a, b \geq 0$ holds, we have 
	\begin{equation*}
	\begin{split}
		&2\left|  \langle \overline{Y}(s), \overline{f} (s) \rangle \right| 
		\\ &\leq  2 \left|  \overline{Y}(s) \right| \left( 
		|\delta_X f (s)|
		 + L_{f, y}|\overline{Y}(s)| + L_{f, z} |\overline{Z}(s)|  \right) 
		\\ &\leq |\delta_{X} f (s)|^2
		+ (1+ L_{f, y}^2 +\epsilon^{-1} )|\overline{Y}(s)| ^2 + 
		\epsilon L_{f, z}^2 |\overline{Z}(s)| ^2 .
	\end{split}
	\end{equation*}
	Since $\overline{Y}, \overline{Z}$ is squared integrable, 
	$ \left\{ \int_{r}^{T}  \langle \overline{Y}(s), \overline{Z}(s) \mathrm{d}W(s) \rangle \right\}_{r \in [t,T] } $ is martingale vanishing at $T$. 
	Therefore, we obtain 
	\begin{equation*}
	\begin{split}
			&\mathbb{E} \left[ \left| \overline{Y} (t) \right|^2 e^{-Lt}\right] 
			+ L \int_t^T \mathbb{E} \left[ \left| \overline{Y} (s) \right|^2 e^{-Ls}\right] \mathrm{d}s
			+  \int_t^T \mathbb{E} \left[ \left| \overline{Z} (s) \right|^2 e^{-Ls}\right] \mathrm{d}s
		\\ &\leq   \mathbb{E} \left[ \left| \delta_{X} \varphi  \right|^2 e^{-LT}\right]  
		\\ &+   \int_{t}^{T}    \mathbb{E} \left[  \left|\delta_{X} f (s) \right|^2 + (1+ L_{f, y}^2 +\epsilon^{-1} ) L_{f, y}^2  \left| \overline{Y} (s) \right|^2 
		 + \epsilon L_{f, z}^2 |\overline{Z}(s)| ^2  \right] e^{-Ls}  \mathrm{d}s.
	\end{split}
	\end{equation*}
	where we write $ \delta_{X} \varphi (T)\triangleq \varphi \left( X_1 (T) \right)-\varphi \left( X_2 (T) \right)$. 
	Now, putting $L = (1+ L_{f, y}^2 +\epsilon^{-1} ) + (1-\epsilon ) $, we obtain 
	\begin{equation*}
	\begin{split}
		&	\mathbb{E} \left[ \left| \overline{Y} (t) \right|^2 e^{-Lt} \right] 
			+(1-\epsilon )\int_t^T \mathbb{E} \left[\left(  \left| \overline{Y} (s) \right|^2+  \left| \overline{Z} (s) \right|^2 \right)  e^{-Ls}\right] \mathrm{d}s
		 \\&\leq   \mathbb{E} \left[ \left| \delta_{X} \varphi (T)\right|^2 e^{-LT}\right]  +  \int_{t}^{T}    \mathbb{E} \left[ \left|\delta_{X} f (s) \right|^2 e^{-Ls} \right]  \mathrm{d}s. 
	\end{split}
	\end{equation*}

	\end{proof}

\subsection*{Acknowledgements}
The authors thank Dai Taguchi for careful reading of the manuscript and 
I would like to thank Yushi Hamaguchi for useful discussions. 
Finally, I would like to thank reviewers for many great comments 
to modify and improve the paper. 


\bibliographystyle{abbrv}

\newpage 

\thispagestyle{empty}

\end{document}